\documentclass[english]{article}
\usepackage[T1]{fontenc}
\usepackage[latin9]{inputenc}
\usepackage{refstyle}
\usepackage{enumitem}
\usepackage{amsmath}
\usepackage{amsthm}
\usepackage{amssymb}
\usepackage{setspace}
\makeatletter

\AtBeginDocument{\providecommand\secref[1]{\ref{sec:#1}}}
\AtBeginDocument{\providecommand\propref[1]{\ref{prop:#1}}}
\AtBeginDocument{\providecommand\enuref[1]{\ref{enu:#1}}}
\AtBeginDocument{\providecommand\remref[1]{\ref{rem:#1}}}
\AtBeginDocument{\providecommand\lemref[1]{\ref{lem:#1}}}
\AtBeginDocument{\providecommand\thmref[1]{\ref{thm:#1}}}
\AtBeginDocument{\providecommand\exaref[1]{\ref{exa:#1}}}
\AtBeginDocument{\providecommand\corref[1]{\ref{cor:#1}}}
\RS@ifundefined{subsecref}
  {\newref{subsec}{name = \RSsectxt}}
  {}
\RS@ifundefined{thmref}
  {\def\RSthmtxt{theorem~}\newref{thm}{name = \RSthmtxt}}
  {}
\RS@ifundefined{lemref}
  {\def\RSlemtxt{lemma~}\newref{lem}{name = \RSlemtxt}}
  {}

\theoremstyle{plain}
\newtheorem{thm}{\protect\theoremname}[section]
  \theoremstyle{definition}
  \newtheorem{defn}[thm]{\protect\definitionname}
  \theoremstyle{plain}
  \newtheorem{prop}[thm]{\protect\propositionname}
  \theoremstyle{plain}
  \newtheorem{lem}[thm]{\protect\lemmaname}
  \theoremstyle{definition}
  \newtheorem{example}[thm]{\protect\examplename}
  \theoremstyle{plain}
  \newtheorem{cor}[thm]{\protect\corollaryname}
  \theoremstyle{remark}
  \newtheorem{rem}[thm]{\protect\remarkname}

\@ifundefined{date}{}{\date{}}
\AtBeginDocument{\providecommand\remref[1]{\ref{rem:#1}}}
\AtBeginDocument{\providecommand\propref[1]{\ref{prop:#1}}}
\RS@ifundefined{propref}{\newref{prop}{name =Proposition~,names = Propositions~}}{}
\AtBeginDocument{\providecommand\corref[1]{\ref{cor:#1}}}
\RS@ifundefined{corref}{\newref{cor}{name =Corollary~,names = Corollaries~}}{}
\RS@ifundefined{exaref}{\newref{exa}{name =Example~,names = Examples~}}{}
\AtBeginDocument{\renewcommand\secref[1]{Section~\ref{sec:#1}}}
\AtBeginDocument{}
\AtBeginDocument{\renewcommand\secref[1]{Section~\ref{sec:#1}}}
\AtBeginDocument{\renewcommand\remref[1]{Remark~\ref{rem:#1}}}

\def\@fnsymbol#1{\ensuremath{\ifcase#1\or *\or **\or \ddagger\or
   \mathsection\or \mathparagraph\or \|\or \dagger\dagger
   \or \ddagger\ddagger \else\@ctrerr\fi}}

\usepackage{enumitem, hyperref}
\makeatletter
\def\namedlabel#1#2{\begingroup
    #2%
    \def\@currentlabel{#2}%
    \phantomsection\label{#1}\endgroup
}

\usepackage{tikz}

\usepackage{tocbibind}
\usepackage{youngtab}
\usepackage{authblk}
\usetikzlibrary{matrix}
\usepackage{bbding}
\usepackage{ytableau}
\usepackage{amsmath}
\usepackage{geometry}
\usepackage{mathtools}

\tikzset{   pt/.style={insert path={node[scale=2]{.}}},   dnup/.style={insert path={ [pt] .. controls +(0,1) and +(0,-1) .. +(#1,2) [pt]}},   dndn/.style={insert path={ [pt] .. controls +(0,0.25) and +(0,0.25) .. +(#1,0) [pt]}},   upup/.style={insert path={ [pt] .. controls +(0,-0.25) and +(0,-0.25) .. +(#1,0) [pt]}}, upup2/.style={insert path={ [pt] .. controls +(0,-0.5) and +(0,-0.5) .. +(#1,0) [pt]}}, }

\newcommand{\Rt}{\widetilde{\mathcal{R}}_E}
\newcommand{\Lt}{\widetilde{\mathcal{L}}_E}

\DeclareMathOperator{\id}{id}

\DeclareMathOperator{\im}{\mathsf{im}}

\DeclareMathOperator{\CT}{C}

\DeclareMathOperator{\Rc}{\mathcal{R}}
\DeclareMathOperator{\Lc}{\mathcal{L}}

\DeclareMathOperator{\Jc}{\mathcal{J}}

\newcommand{\PCS}{\text{PCS}}
\newcommand{\MCS}{\text{MCS}}

\def\RSlemtxt{Lemma~}
\def\RSthmtxt{Theorem~}

\usepackage{babel}
\providecommand{\corollaryname}{Corollary}
  \providecommand{\definitionname}{Definition}
  \providecommand{\examplename}{Example}
  \providecommand{\lemmaname}{Lemma}
  \providecommand{\propositionname}{Proposition}
  \providecommand{\remarkname}{Remark}
 
\providecommand{\theoremname}{Theorem}

\makeatother

\usepackage{babel}
  \providecommand{\corollaryname}{Corollary}
  \providecommand{\definitionname}{Definition}
  \providecommand{\examplename}{Example}
  \providecommand{\lemmaname}{Lemma}
  \providecommand{\propositionname}{Proposition}
  \providecommand{\remarkname}{Remark}
\providecommand{\theoremname}{Theorem}

\begin{document}

\title{Algebras of Reduced $E$-Fountain Semigroups and the Generalized
Ample Identity}

\author{Itamar Stein\thanks{Mathematics Unit, Shamoon College of Engineering, 77245 Ashdod, Israel }\\
\Envelope \, Steinita@gmail.com}
\maketitle
\begin{abstract}
Let $S$ be a reduced $E$-Fountain semigroup. If $S$ satisfies the
congruence condition, there is a natural construction of a category
$\mathcal{C}$ associated with $S$. We define a $\Bbbk$-module homomorphism
$\varphi:\Bbbk S\to\Bbbk\mathcal{C}$ (where $\Bbbk$ is any unital
commutative ring). With some assumptions, we prove that $\varphi$
is an isomorphism of $\Bbbk$-algebras if and only if some weak form
of the right ample identity holds in $S$. This gives a unified generalization
for a result of the author on right restriction $E$-Ehresmann semigroups
and a result of Margolis and Steinberg on the Catalan monoid.
\end{abstract}


\section{Introduction }

Let $S$ be a semigroup and let $\Bbbk$ be a commutative unital ring.
It is of interest to study the semigroup algebra $\Bbbk S$ and its
representations. Quite often it is complicated to investigate the
algebra $\Bbbk S$ directly, but a change of basis yields an isomorphic
algebra whose structure is more transparent. A pioneering work in
this direction was done by Solomon \cite{Solomon1967} who proved
that the semigroup algebra of a finite semilattice is isomorphic to
a product of copies of the base ring. This was generalized by Steinberg
\cite{Steinberg2006} who showed that the algebra of any finite inverse
semigroup is isomorphic to the algebra of its associated inductive
groupoid. Steinberg's result is now fundamental in the study of representations
of finite inverse semigroups. Guo and Chen \cite{Guo2012} obtained
a similar result for finite ample semigroups and the author extended
this generalization to a class of right restriction $E$-Ehresmann
semigroups \cite{Stein2017,Stein2018erratum} ($E$-Ehresmann semigroups
were introduced by Lawson in \cite{Lawson1991}). This result has
led to several applications regarding semigroups of partial functions
\cite{Stein2016,Stein2019,Stein2020,Margolis2021} and recently also
to the study of certain partition monoids \cite{East2021}. We mention
also that Wang \cite{Wang2017} generalized the above results further
to a certain class of right $P$-restriction, $P$-Ehresmann semigroups
(for definitions of these notions see \cite{Jones2012}) - but we
do not follow this approach in this paper. A hint for another direction
is given by the Catalan monoid. The Catalan monoid $\CT_{n}$ contains
the order-preserving $(x\leq y\implies f(x)\leq f(y)$) and order-increasing
($x\leq f(x)$) functions $f$ on $\{1,\ldots,n\}$. It is known that
the algebra of the Catalan monoid is isomorphic to a certain incidence
algebra (\cite[Theorem 5.5]{Hivert2009} and \cite[Theorem 17.25]{Steinberg2016})
but recently Margolis and Steinberg obtained a simple proof of this
fact using the change of basis approach \cite{Margolis2018A}. Their
result is not implied by any of the above-mentioned results. The goal
of this paper is to obtain a generalization for the theorem on right
restriction $E$-Ehresmann semigroups that includes also the case
of the Catalan monoid. The class of semigroups which provides the
correct context for this task is the class of reduced $E$-Fountain
semigroups - also introduced by Lawson \cite{Lawson1990} under the
name reduced $E$-semiabundant semigroups (which also appears in the
literature as DR-semigroups \cite{Stokes2015}). Given a subset of
idempotents $E$ of $S$ we can define two equivalence relations $\Lt$
and $\Rt$ on $S$. We say that $a\Lt b$ ($a\Rt b$) if $a$ and
$b$ have the same set of right (respectively, left) identities from
$E$. The semigroup $S$ is called reduced $E$-Fountain if every
$\Lt$ and $\Rt$-class contains a (unique) idempotent from $E$ and
$ef=e\iff fe=e$ for every $e,f\in E$. If in addition $\Lt$ and
$\Rt$ are right and left congruences respectively then we can associate
a certain category $\mathcal{C}={\bf C}(S)$ with the semigroup $S$
(for full details see \cite{Lawson1991}). We remark that $E$-Ehresmann
semigroups are precisely those reduced $E$-Fountain semigroups which
satisfy the congruence condition and whose distinguished subset of
idempotents $E$ forms a subsemilattice (i.e., a commutative subsemigroup
of idempotents). For $a,b\in S$ we define a relation $\trianglelefteq_{l}$
by the rule that $a\trianglelefteq_{l}b$ if $a=be$ for an idempotent
$e\in E$. This is a generalization of the right restriction partial
order defined on $E$-Ehresmann semigroups - but $\trianglelefteq_{l}$
is not a partial order or even a preorder. We also define an identity
we call the generalized right ample condition which is a weak form
of the right ample condition studied in the theory of $E$-Fountain
semigroups. This background is described in \secref{Preliminaries}
and \secref{Theory_of_idempotents} of the paper. Let $S$ be a reduced
$E$-Fountain semigroup which satisfies the congruence condition and
assume also that sets of the form $\{b\in S\mid b\trianglelefteq_{l}a\}$
are finite for every $a\in S$. Let $\Bbbk$ be a commutative unital
ring. In \secref{Algebras_Section} we define a ``change of basis''
$\Bbbk$-module homomorphism $\varphi$ between the semigroup algebra
$\Bbbk S$ and $\Bbbk\mathcal{C}$ - the algebra of the associated
category $\mathcal{C}$. We show that $\varphi$ is a homomorphism
of $\Bbbk$-algebras if and only if the generalized right ample identity
holds. We also obtain an isomorphism in case $\trianglelefteq_{l}$
is contained in a partial order. In \secref{Ehresmann_semigroups}
we consider the case where $E$ is a subband (i.e., a subsemigroup
of idempotents). In this case, we obtain an $E$-Ehresmann semigroup
and $\trianglelefteq_{l}$ is a partial order. Moreover, the generalized
right ample condition and the standard one coincide. In \secref{Catalan_monoid}
we discuss in detail the case of the Catalan monoid and show that
we retrieve the isomorphism described in \cite{Margolis2018A}.

\textbf{Acknowledgments:} The author thanks Professor Victoria Gould
for a helpful conversation and the referee for helpful comments.

\section{\label{sec:Preliminaries}Preliminaries}

Let $S$ be a semigroup and let $S^{1}=S\cup\{1\}$ be the monoid
formed by adjoining a formal unit element. Recall that Green's preorders
$\leq_{\Rc}$, $\leq_{\Lc}$ and $\leq_{\Jc}$ are defined by:
\begin{align*}
a\leq_{\Rc}b\iff & aS^{1}\subseteq bS^{1}\\
a\leq_{\Lc}b\iff & S^{1}a\subseteq S^{1}b\\
a\leq_{\Jc}b\iff & S^{1}aS^{1}\subseteq S^{1}bS^{1}
\end{align*}
The associated Green's equivalence relations on $S$ are denoted by
$\Rc$, $\Lc$ and $\Jc$. It is well known that $\Lc$ ($\Rc$) is
a right congruence (respectively, left congruence). A semigroup $S$
is called $\Jc$- trivial if $\Jc$ is the identity relation (that
is, $a\Jc b\iff a=b$). Similar definitions hold for $\Rc$-trivial
and $\Lc$-trivial semigroups. We denote by $E(S)$ the set of idempotents
of $S$. We denote by $\leq$ the natural partial order on $E(S)$
defined by 
\[
e\leq f\iff(ef=fe=e).
\]

Other elementary semigroup theoretic notions can be found in \cite{Howie1995}.
Let $E\subseteq E(S)$ be some subset of idempotents. For every $a\in S$
we denote by $a_{E}$ the set of right identities of $a$ from $E$:
\[
a_{E}=\{e\in E\mid ae=a\}
\]
Dually $\prescript{}{E}{a}$ is the set of left identities from $E$.
We define two equivalence relations $\Lt$ and $\Rt$ on $S$ by 
\begin{align*}
a\Lt b & \iff(a_{E}=b_{E})\\
a\Rt b & \iff(\ensuremath{\prescript{}{E}{a}}=\ensuremath{\prescript{}{E}{b}}).
\end{align*}

It is easy to see that $\Lc\subseteq\Lt$ and $\Rc\subseteq\Rt$.
Recall that a subsemigroup of idempotents $E\subseteq S$ is called
a \emph{subband.} A commutative subband is called a \emph{subsemilattice}.

Let $\Bbbk$ be a commutative unital ring. The \emph{semigroup algebra}
$\mathbb{\Bbbk}S$ of a semigroup $S$ is defined in the following
way. It is a free $\mathbb{\Bbbk}$-module with basis the elements
of $S$, that is, it consists of all formal linear combinations
\[
\{k_{1}s_{1}+\ldots+k_{n}s_{n}\mid k_{i}\in\mathbb{\Bbbk},\,s_{i}\in S\}.
\]
The multiplication in $\mathbb{\Bbbk}S$ is the linear extension of
the semigroup multiplication. We will also need the notion of a category
algebra in this paper. The \emph{category algebra} $\mathbb{\Bbbk}\mathcal{C}$
of a (small) category $\mathcal{C}$ is defined in the following way.
It is a free $\mathbb{\Bbbk}$-module with the morphisms of $\mathcal{C}$
as a basis, that is, it consists of all formal linear combinations
\[
\{k_{1}m_{1}+\ldots+k_{n}m_{n}\mid k_{i}\in\mathbb{\Bbbk},\,m_{i}\in\mathcal{C}^{1}\}.
\]
The multiplication in $\mathbb{\Bbbk}\mathcal{C}$ is the linear extension
of the following:
\[
m^{\prime}\cdot m=\begin{cases}
m^{\prime}m & \text{if \ensuremath{m^{\prime}m} is defined}\\
0 & \text{otherwise}.
\end{cases}
\]
Let $R$ be a relation on a set $X$. We say that $R$ is principally
finite if the set $\{x^{\prime}\in X\mid x^{\prime}Rx\}$ is finite
for every $x\in X$. If $R=\preceq$ is a principally finite partial
order, we can define the \emph{incidence algebra of $\preceq$ }which\emph{
}consists of all functions $f:\preceq\to\Bbbk$ with standard addition
operation and whose multiplication operation is defined by 
\[
f\star g(a,b)=\sum_{a\preceq c\preceq b}f(a,c)g(c,b).
\]
We denote this algebra by $\Bbbk[\preceq]$. It is well known that
an element $f\in\Bbbk[\preceq]$ is invertible if and only if $f(x,x)$
is invertible in $\Bbbk$ for every $x\in X$. For other elementary
facts on incidence algebras see \cite[Section 3.6]{Stanley1997}).
Another point is worth mentioning. The poset $\preceq$ can also be
viewed as a category whose set of objects is $X$ and there exists
a unique morphism from an object $x_{1}$ to another object $x_{2}$
if and only if $x_{1}\preceq x_{2}$. The category algebra in general
differs from the incidence algebra because the elements of the category
algebra are only \emph{finite} linear combinations. However, the two
notions coincide if $X$ is a finite set.

\section{\label{sec:Theory_of_idempotents}Reduced $E$-Fountain Semigroups}

\subsection{Basic definitions}
\begin{defn}
A semigroup $S$ is called \emph{$E$-Fountain} if every $\Lt$-class
contains an idempotent from $E$ and every $\Rt$-class contains an
idempotent from $E$, where $E\subseteq S$ is some subset of idempotents.
We remark that this property is also called ``$E$-semiabundant''
in the literature.
\end{defn}
The following is an immediate consequence of \cite[Propositions 1.2 and 1.3]{Stokes2015}.
\begin{prop}
\label{prop:Reduced_E_Fountain}Let $S$ be an $E$-Fountain semigroup.
The following conditions are equivalent:
\begin{enumerate}
\item $\forall e,f\in E$,$\quad$ $ef=e\iff fe=e$.
\item For every $a\in S$ the sets $a_{E}$ and $\prescript{}{E}{a}$ contain
a minimum element (with respect to the natural partial order on idempotents).
\item \label{enu:Variety_of_reduced_E_fountain}We can equip $S$ with two
unary operations $\ast$ and $+$ which satisfy the following identities
for every $a,b\in S$:
\begin{align*}
a^{+}a & =a\quad(a^{+})^{+}=a^{+}\quad a{}^{+}(ab)^{+}=(ab)^{+}a^{+}=(ab)^{+}\\
aa^{\ast} & =a\quad(a^{\ast})^{\ast}=a^{\ast}\quad b{}^{\ast}(ab)^{\ast}=(ab)^{\ast}b^{\ast}=(ab)^{\ast}\\
(a^{+})^{\ast} & =a^{+}\quad(a^{\ast})^{+}=a^{\ast}
\end{align*}
\end{enumerate}
\end{prop}
It is important to note that $a^{\ast}$ ($a^{+})$ is the minimum
element of $a_{E}$ (respectively, $\prescript{}{E}{a}$) and that
$e^{\ast}=e^{+}=e$ for every $e\in E$.

We follow \cite{Lawson1990} and call an $E$-Fountain semigroup \emph{reduced}
if it satisfies the equivalent conditions of \propref{Reduced_E_Fountain}.
Such a semigroup is called a ``DR-semigroup'' in \cite{Stokes2015}.
It is clear from Condition \enuref{Variety_of_reduced_E_fountain}
that the class of all reduced $E$ -Fountain semigroup is a variety
of bi-unary semigroups.
\begin{defn}
Let $S$ be a reduced $E$-Fountain semigroup. We say that $S$ satisfies
the \emph{congruence condition} if $\Lt$ is a right congruence and
$\Rt$ is a left congruence. 
\end{defn}
It is well known that $S$ satisfies the congruence condition if and
only if the identities $(ab)^{\ast}=(a^{\ast}b)^{\ast}$ and $(ab)^{+}=(ab^{+})^{+}$
hold - see \cite[Lemma 4.1]{Gould2010}. In this case we can define
a category $\mathcal{C}(S)$ in the following way. The objects are
in one-to-one correspondence with the set $E$. The morphisms are
in one to one correspondence with elements of $S$. The convention
is to compose morphisms in such categories ``from left to right''.
However, for us it will be more convenient to use composition ``from
right to left''. Therefore, for every $a\in S$ the associated morphism
$C(a)$ has domain $a^{\ast}$ and range $a^{+}$. If the range of
$C(a)$ is the domain of $C(b)$ (that is, if $b^{\ast}=a^{+}$) the
composition $C(b)\cdot C(a)$ is defined to be $C(ba)$. The assumption
$b^{\ast}=a^{+}$ implies that $(ba)^{+}=(ba^{+})^{+}=(bb^{\ast})^{+}=b^{+}$
and likewise $(ba)^{\ast}=a^{\ast}$ so this is indeed a category
- see \cite{Lawson1991} for additional details. 

\subsection{The relations $\leq_{l}$ and $\trianglelefteq_{l}$}

Let $S$ be a reduced $E$-Fountain semigroup. The following partial
order is defined in \cite{Lawson1991}:
\[
a\leq_{l}b\iff(a^{\ast}\leq b^{\ast}\text{ and }a=ba^{\ast})
\]
This is a generalization of the right ``restriction'' partial order
of an $E$-Ehresmann semigroup and in particular of the natural partial
order of an inverse semigroup. However, we will need a different generalization
for our purpose.
\begin{defn}
Let $a,b\in S$. We define $a\trianglelefteq_{l}b\iff a=be$ for some
$e\in E$.
\end{defn}
\begin{lem}
$a\trianglelefteq_{l}b\iff a=ba^{\ast}$.
\end{lem}
\begin{proof}
The $\Leftarrow$ direction is immediate. For the other direction
assume $a=be$ so $ae=a$ hence $a^{\ast}\leq e$ in $E$. Therefore
\[
a=aa^{\ast}=a(ea^{\ast})=(ae)a^{\ast}=ba^{\ast}
\]
as required.
\end{proof}
The following uniqueness property will also be useful.
\begin{lem}
\label{lem:Uniqueness_of_restriction} If $a\trianglelefteq_{l}b$
then $a$ is the unique $x\in S$ such that $x\trianglelefteq_{l}b$
and $x^{\ast}=a^{\ast}$.
\end{lem}
\begin{proof}
Assume $x\trianglelefteq_{l}b$ such that $x^{\ast}=a^{\ast}$. Then
$x=bx^{\ast}=ba^{\ast}=a$.
\end{proof}
It is clear that $\leq_{l}\subseteq\trianglelefteq_{l}$ and it is
easy to see that $e\leq f\iff e\leq_{l}f\iff e\trianglelefteq_{l}f$
for $e,f\in E$. It is tempting to think that $\trianglelefteq_{l}$
is also a partial order. It is clear that $\trianglelefteq_{l}$ is
reflexive since $a=aa^{\ast}$. However, the following example shows
that in general it is not antisymmetric and in \remref{Leq_l_not_transitive}
we will give a counter-example for transitivity.
\begin{example}
\cite[Example 2.3]{Wang2017}\label{exa:Square_rectangular_band}
Let $S$ be a ``square'' rectangular band. The elements of $S$
are pairs $(i,j)$ where $1\leq i,j\leq n$ for some fixed $n\in\mathbb{N}$.
Multiplication is defined by 
\[
(i_{1},j_{1})\cdot(i_{2},j_{2})=(i_{1},j_{2}).
\]
Recall that $(i_{1},j_{1})\Rc(i_{2},j_{2})\iff i_{1}=i_{2}$ and $(i_{1},j_{1})\Lc(i_{2},j_{2})\iff j_{1}=j_{2}$.
Every element of $S$ is an idempotent, but we choose $E$ to be the
set of ``diagonal'' elements $E=\{(i,i)\mid1\leq i\leq n\}$. It
is easy to see that 
\[
(i,j)e=(i,j)\iff e=(j,j)
\]
 for $e\in E$. Therefore, the set of right identities of $(i,j)$
contains a unique idempotent from $E$ and the dual claim holds for
left identities. Hence, $S$ is a reduced $E$-Fountain semigroup.
It is also easy to check that $(i_{1},j_{1})\trianglelefteq_{l}(i_{2},j_{2})$
if and only if $i_{1}=i_{2}$ . Therefore, $a\trianglelefteq_{l}b$
if and only if $a\Rc b$ so $\trianglelefteq_{l}$ is in fact a symmetric
relation! On the other hand, $\leq$ is the trivial relation on $E$
so $\leq_{l}$ is also trivial on $S$. We give another simple observation
for future use. It is easy to verify that $(i_{1},j_{1})\Lt(i_{2},j_{2})$
if and only if $j_{1}=j_{2}$. Therefore, $\Lt=\Lc$ so $\Lt$ is
a right congruence and and dually $\Rt=\Rc$ is a left congruence
so $S$ satisfies also the congruence condition.
\end{example}
Clearly, we can also define a dual relation $a\trianglelefteq_{r}b\iff a=eb$
for some $e\in E$.

\subsection{The right ample and the generalized right ample conditions}

Let $S$ be a reduced $E$-Fountain semigroup which satisfies the
congruence condition. We say that the \emph{right ample condition}
(or \emph{right ample identity}) holds in $S$ if $ea=a(ea)^{\ast}$
for every $a\in S$ and $e\in E$. This is equivalent to the condition
$Ea\subseteq aE$ for every $a\in S$. The right ample condition is
well studied, but it is a too strong requirement for some of the monoids
we want to consider as we immediately show. 
\begin{lem}
\label{lem:RightAmpleImpliesSubband}If the right ample condition
holds then $E$ is a subband of $S$. 
\end{lem}
\begin{proof}
Let $e,f\in E$. First note that $(ef)^{\ast}\leq f$ since $eff=ef$.
Now, the right ample identity implies that 
\[
ef=f(ef)^{\ast}=(ef)^{\ast}\in E
\]
so $E$ is a subband.
\end{proof}
\begin{lem}
\label{lem:SubbandImpliesEhresmann}Let $S$ be a reduced $E$-Fountain
semigroup. If $E$ is a subband of $S$ then $E$ is a subsemilattice.
\end{lem}
\begin{proof}
Let $e,f\in E$. The fact that $e(ef)=ef$ implies that $(ef)e=ef$.
Likewise, $(fe)e=fe$ implies $efe=fe$ so $ef=fe$.
\end{proof}
\begin{defn}
A reduced $E$-Fountain semigroup which satisfies the congruence condition
is called \emph{$E$-Ehresmann }\cite{Gould2010,Gould2010b} if $E$
is a subsemilattice of $S$. 
\end{defn}
If we consider reduced $E$-Fountain semigroups, \lemref{RightAmpleImpliesSubband}
and \lemref{SubbandImpliesEhresmann} show that the right ample condition
is relevant only in the context of $E$-Ehresmann semigroups. We want
a condition which is applicable also for reduced $E$-Fountain semigroups
where $E$ is not a subband.
\begin{defn}
Let $S$ be a reduced $E$-Fountain semigroup which satisfies the
congruence condition. We say that the \emph{generalized right ample
condition }(or \emph{generalized right ample identity})\emph{ }holds
in $S$ if 
\[
\left(e\left(a\left(eaf\right)^{\ast}\right)^{+}\right)^{\ast}=\left(a\left(eaf\right)^{\ast}\right)^{+}
\]
for every $a\in S$ and $e,f\in E$.
\end{defn}
This identity can also be written as 
\[
\left(b^{\ast}\left(a\left(b^{\ast}ac^{\ast}\right)^{\ast}\right)^{+}\right)^{\ast}=\left(a\left(b^{\ast}ac^{\ast}\right)^{\ast}\right)^{+}
\]
for every $a,b,c\in S$ hence the class of reduced $E$-Fountain semigroups
which satisfy the congruence condition and the generalized right ample
identity is also a variety of bi-unary semigroups. 

Next, we show that if $S$ satisfies the right ample condition it
satisfies also the generalized right ample condition hence justifying
our term ``generalized''. We fix a reduced $E$-Fountain semigroup
$S$ which satisfies the congruence condition.
\begin{lem}
\label{lem:Right_restriction_aux1}If $E$ is a subsemilattice then
$a\trianglelefteq_{l}b\implies a^{\ast}\leq b^{\ast}$.
\end{lem}
\begin{proof}
If $a=ba^{\ast}$ then $ab^{\ast}=ba^{\ast}b^{\ast}=bb^{\ast}a^{\ast}=ba^{\ast}=a$
so $a^{\ast}\leq b^{\ast}$.
\end{proof}
We record the following immediate corollary.
\begin{cor}
\label{cor:EqualityOfLeqs}If $E$ is a subsemilattice then $\trianglelefteq_{l}=\leq_{l}$.
\end{cor}
\begin{prop}
\label{prop:Standard_right_ample_implies_generalised}If the right
ample condition holds in $S$ then so does the generalized right ample
condition.
\end{prop}
\begin{proof}
Let $e,f\in E$ and $a\in S$. By \lemref{RightAmpleImpliesSubband}
and \lemref{SubbandImpliesEhresmann} we know that $E$ is a semilattice.
Since $eaf\trianglelefteq_{l}ea$, \lemref{Right_restriction_aux1}
implies that $(ea)^{\ast}(eaf)^{\ast}=(eaf)^{\ast}$. This and the
right ample identity implies that 
\[
ea(eaf)^{\ast}=a(ea)^{\ast}(eaf)^{\ast}=a(eaf)^{\ast}
\]
so $e$ is a left identity of $a(eaf)^{\ast}$. Since $a(eaf)^{\ast}\Rt(a(eaf)^{\ast})^{+}$
we obtain $e(a(eaf)^{\ast})^{+}=(a(eaf)^{\ast})^{+}$ as well. Finally,
\[
\left(e\left(a\left(eaf\right)^{\ast}\right)^{+}\right)^{\ast}=\left(\left(a\left(eaf\right)^{\ast}\right)^{+}\right)^{\ast}=\left(a\left(eaf\right)^{\ast}\right)^{+}
\]
 so $S$ satisfies the generalized right ample condition.
\end{proof}
\begin{rem}
We can of course define the dual notion. We say that $S$ satisfies
the \emph{generalized left ample condition }if the identity
\[
\left(\left(\left(fae\right)^{+}a\right)^{\ast}e\right)^{+}=\left(\left(fae\right)^{+}a\right)^{\ast}
\]
holds for every $a\in S$ and $e,f\in E$. 
\end{rem}
We give here two examples of semigroups which satisfy the generalized
right ample condition but not the standard one. Although the main
theorem of this paper (\thmref{isomorphism_theorem}) is not applicable
in these examples as we will see later, we hope they offer conviction
that the generalized right ample identity is not an artificial condition.
The main motivating example - the Catalan monoid - is postponed until
\secref{Catalan_monoid}.
\begin{example}
\label{exa:Rectangular_band_is_weak_right_restriction}The rectangular
band from \exaref{Square_rectangular_band} satisfies the generalized
right ample condition but not the right ample condition. If $\left(a\left(eaf\right)^{\ast}\right)^{+}=(i,i)$
and $e=(k,k)$ then 
\[
\left(e\left(a\left(eaf\right)^{\ast}\right)^{+}\right)^{\ast}=\left((k,k)(i,i)\right)^{\ast}=(k,i)^{\ast}=(i,i)=\left(a\left(eaf\right)^{\ast}\right)^{+}
\]
so the required equality holds for every $e,f\in E$, and $a\in S$.
On the other hand, $E$ is not a band so the right ample condition
doesn't hold. We remark that it is equally easy to see that the left
generalized ample identity holds in this case.
\end{example}
\begin{example}[\cite{Stokes2015}]
 Let $H$ be a Hilbert space and let $L(H)$ be the algebra of all
bounded linear operators on $H$ (this is one of the main examples
of a Rickart $\ast$-ring and a Baer $\ast$-ring \cite{berberian1988baer}).
For every closed subspace $U\subseteq H$ we associate the orthogonal
projection $P_{U}\in L(H)$ onto it. Recall that $P_{U}$ is an idempotent
and $\im(P_{U})=U=\ker(P_{U})^{\perp}$ (where $\im T$ and $\ker T$
are the image and kernel of the linear operator $T$ and $V^{\perp}$
is the orthogonal complement of $V$). It is easy to check that 
\[
P_{U}P_{V}=P_{U}\iff P_{V}P_{U}=P_{U}\iff U\subseteq V.
\]
Now, consider $L(H)$ as a multiplicative monoid. This is an example
of a Baer $\ast$-semigroup \cite{Foulis1960}. If we set $E=\{P_{U}\mid U\text{ is a closed subspace of }H\}$,
it is not difficult to check that for $T,S\in L(H)$
\[
T\Lt S\iff\ker(T)=\ker(S),\qquad T\Rt S\iff\im(T)=\im(S).
\]
It follows that $L(H)$ is a reduced $E$-Fountain semigroup where
$T^{\ast}=P_{(\ker T)^{\perp}}$ and $T^{+}=P_{\im(T)}$.
\end{example}
\begin{lem}
$L(H)$ satisfies the congruence condition.
\end{lem}
\begin{proof}
Let $T,S,R\in L(H)$. First assume $T\Lt S$ so $\ker(T)=\ker(S)$.
We have 
\[
x\in\ker(TR)\iff R(x)\in\ker T\iff R(x)\in\ker S\iff x\in\ker(SR)
\]
so $\ker(TR)=\ker(SR)$ and $\Lt$ is a right congruence. If $T\Rt S$
then $\im(T)=\im(S)$. In this case 
\[
x\in\im(RT)\iff\exists y\in\im(T),\quad R(y)=x\iff\exists y\in\im(S),\quad R(y)=x\iff x\in\im(RS)
\]
so $\Rt$ is a left congruence.
\end{proof}
Note that in general $P_{U}P_{V}\neq P_{V}P_{U}$ so the right ample
identity does not hold and $L(H)$ is not an $E$-Ehresmann semigroup.
\begin{lem}
$L(H)$ satisfies the generalized right ample identity.
\end{lem}
\begin{proof}
Let $a\in L(H)$ and $e,f\in E$. Set $e=P_{U}$, $(eaf)^{\ast}=P_{W}$
and $\left(a(eaf)^{\ast}\right)^{+}=P_{V}$. We need to show that
$(P_{U}P_{V})^{\ast}=P_{V}$, which means that $\ker(P_{U}P_{V})=\ker(P_{V})$.
This is equivalent to $\im(P_{V})\cap\ker(P_{U})=0$ so we need to
prove that $V\cap U^{\perp}=0$ where $V=\im\left(a(P_{U}af)^{\ast}\right)$.
Let $\mbox{\ensuremath{y\in V\cap U^{\perp}}}$, then $y=a(x)$ for
$x\in\im(P_{U}af)^{\ast}=W=\ker(P_{U}af)^{\perp}$. The fact that
$\ker(f)\subseteq\ker(P_{U}af)$ implies that $W=\ker(P_{U}af)^{\perp}\subseteq\ker(f)^{\perp}$.
Since $f\in E$, $\ker(f)^{\perp}=\im(f)$ so $f(x)=x$ and 
\[
P_{U}af(x)=P_{U}a(x)=P_{U}(y)=0
\]
since $y\in U^{\perp}.$ This shows that $x\in\ker(P_{U}af)$, but
we already know that $x\in\ker(P_{U}af)^{\perp}$ so $x=0$ and $y=a(x)=0$
as required.
\end{proof}
\begin{rem}
$L(H)$ satisfies also the left ample identity. For $T\in L(H)$ we
denote by $T^{t}$ its adjoint operator (we do not use $T^{\ast}$
here to avoid confusion with the unary operation). It is clear that
$(P_{U})^{t}=P_{U}$ , $(T^{t})^{+}=T^{\ast}$ and $(T^{t})^{\ast}=T^{+}$
so every identity immediately implies its dual in $L(H)$.
\end{rem}
This example will not be considered further in this paper because
it does not satisfy the finiteness condition we will require later.

\section{\label{sec:Algebras_Section}The semigroup and category algebras}

\subsection{Homomorphism of algebras}

In this section we fix a reduced $E$-Fountain semigroup $S$ which
satisfies the congruence condition. We also assume that the relation
$\trianglelefteq_{l}$ is principally finite. Let $\mathcal{C}$ be
the associated category as defined in \secref{Theory_of_idempotents}
and let $\Bbbk$ be a commutative unital ring. Define $\varphi:\Bbbk S\to\Bbbk\mathcal{C}$
on basis elements by 
\[
\varphi(a)=\sum_{c\trianglelefteq_{l}a}C(c).
\]

It is clear that $\varphi$ is a $\Bbbk$-module homomorphism. We
want to show that it is an algebra homomorphism if and only if the
generalized right ample condition holds in $S$. For this we need
a different formulation of this identity. 
\begin{lem}
\label{lem:WRR_equivalent_definitions}The following conditions are
equivalent.
\begin{enumerate}
\item The semigroup $S$ satisfies the generalized right ample condition.
\item The implication $c\trianglelefteq_{l}ea\implies(e(ac^{\ast})^{+})^{\ast}=(ac^{\ast})^{+}$
holds for every $a,c\in S$ and $e\in E$.
\item \label{enu:WRR_def_3}The implication $c\trianglelefteq_{l}ba\implies(b(ac^{\ast})^{+})^{\ast}=(ac^{\ast})^{+}$
holds for every $\mbox{\ensuremath{a,b,c\in S}}$.
\end{enumerate}
\end{lem}
\begin{proof}
$\quad$
\begin{itemize}
\item[$(1\implies2)$]  If $c\trianglelefteq_{l}ea$ then $c=eaf$ for some $f\in E$. The
equality $(e(ac^{\ast})^{+})^{\ast}=(ac^{\ast})^{+}$ follows immediately
from the generalized right ample identity.
\item[$(2\implies3)$]  Assume $c\trianglelefteq_{l}ba$ so $c=bac^{\ast}$. Denote $d=b^{\ast}ac^{\ast}$
so clearly $d\trianglelefteq_{l}b^{\ast}a$. Our assumption implies
\[
(b^{\ast}(ad^{\ast})^{+})^{\ast}=(ad^{\ast})^{+}.
\]
Now,  
\[
d^{\ast}=(b^{\ast}ac^{\ast})^{\ast}=(bac^{\ast})^{\ast}=c^{\ast}
\]
so the above equality can be written 
\[
(b^{\ast}(ac^{\ast})^{+})^{\ast}=(ac^{\ast})^{+}.
\]
Therefore, 
\[
(b(ac^{\ast})^{+})^{\ast}=(b^{\ast}(ac^{\ast})^{+})^{\ast}=(ac^{\ast})^{+}
\]
which finishes the proof.
\item[$(3\implies1)$] First substitute $b=e$ for $e\in E$ and then note that $eaf\trianglelefteq_{l}ea$
so we can choose $c=eaf$ and the assumption implies 
\[
\left(e\left(a\left(eaf\right)^{\ast}\right)^{+}\right)^{\ast}=\left(a\left(eaf\right)^{\ast}\right)^{+}
\]
as required. 
\end{itemize}
\end{proof}
\begin{thm}
\label{thm:main_thm_for_now}The module homomorphism $\varphi$ is
a homomorphism of $\Bbbk$-algebras if and only if the generalized
right ample identity holds in $S$.
\end{thm}
\begin{proof}
We start with the ``if'' part. For $a,b\in S$, we need to show
that $\varphi(ba)=\varphi(b)\varphi(a)$. We need to prove that
\[
\sum_{c\trianglelefteq_{l}ba}C(c)=\sum_{c^{\prime\prime}\trianglelefteq_{l}b}C(c^{\prime\prime})\sum_{c^{\prime}\trianglelefteq_{l}a}C(c^{\prime}).
\]
First note that if $c^{\prime\prime}\trianglelefteq_{l}b$ and $c^{\prime}\trianglelefteq_{l}a$
and $C(c^{\prime\prime})C(c^{\prime})\neq0$ then $(c^{\prime\prime})^{\ast}=(c^{\prime})^{+}$
and so 
\[
c^{\prime\prime}c^{\prime}=b(c^{\prime\prime})^{\ast}c^{\prime}=b(c^{\prime})^{+}c^{\prime}=bc^{\prime}=ba(c^{\prime})^{\ast}
\]
 so $c^{\prime\prime}c^{\prime}\trianglelefteq_{l}ba$. Therefore,
every element on the right-hand side appears also on the left-hand
side. Now, take $c\trianglelefteq_{l}ba$ and define $c^{\prime}=ac^{\ast}$
and $c^{\prime\prime}=b(ac^{\ast})^{+}$. From the generalized right
ample condition we obtain
\[
(c^{\prime\prime})^{\ast}=(b(ac^{\ast})^{+})^{\ast}=(ac^{\ast})^{+}=(c^{\prime})^{+}.
\]

Therefore, the composition $C(c^{\prime\prime})\cdot C(c^{\prime})$
is defined. Now 
\[
c^{\prime\prime}c^{\prime}=b(ac^{\ast})^{+}ac^{\ast}=bac^{\ast}=c
\]
so every element from the left-hand side appears on the right-hand
side. It remains to show that it appears only once. Assume $C(c)=C(d^{\prime\prime})C(d^{\prime})$
for $d^{\prime\prime}\trianglelefteq_{l}b$ and $d^{\prime}\trianglelefteq_{l}a$.
Then $(d^{\prime})^{\ast}=c^{\ast}$ so by \lemref{Uniqueness_of_restriction}
we must have $d^{\prime}=ac^{\ast}=c^{\prime}$. Since $C(d^{\prime\prime})\cdot C(d^{\prime})$
is defined we must have $(d^{\prime\prime})^{\ast}=(d^{\prime})^{+}=(ac^{\ast})^{+}$
so again \lemref{Uniqueness_of_restriction} implies $d^{\prime\prime}=b(ac^{\ast})^{+}=c^{\prime\prime}$.
This proves uniqueness.

The ``only if'' part comes from another examination of the above
argument. Let $a,b\in S$ , if $\varphi$ is a homomorphism of algebras
we must have $\varphi(ba)=\varphi(b)\varphi(a)$ so
\[
\sum_{c\trianglelefteq_{l}ba}C(c)=\sum_{c^{\prime\prime}\trianglelefteq_{l}b}C(c^{\prime\prime})\sum_{c^{\prime}\trianglelefteq_{l}a}C(c^{\prime}).
\]
Choose any $c\trianglelefteq_{l}ba$ from the left-hand side. The
morphism $C(c)$ must appear on the right-hand side. So there exists
$c^{\prime}\trianglelefteq_{l}a$ and $c^{\prime\prime}\trianglelefteq_{l}b$
such that $C(c)=C(c^{\prime\prime})C(c^{\prime})$. This implies $(c^{\prime})^{\ast}=c^{\ast}$
so $c^{\prime}=ac^{\ast}$. Since the product $C(c)=C(c^{\prime\prime})C(c^{\prime})$
is defined it must be the case that $(c^{\prime\prime})^{\ast}=(ac^{\ast})^{+}$
and therefore $c^{\prime\prime}=b(ac^{\ast})^{+}$. This implies $(b(ac^{\ast})^{+})^{\ast}=(ac^{\ast})^{+}$
as required.
\end{proof}

\subsection{Isomorphism of algebras}

In general, $\varphi$ is not an isomorphism. For instance, consider
\exaref{Square_rectangular_band}. We have already seen (\exaref{Rectangular_band_is_weak_right_restriction})
that it satisfies the requirements of \thmref{main_thm_for_now} so
$\varphi$ is an homomorphism of $\Bbbk$-algebras. But in this example
we have $\varphi(a)=\varphi(b)$ if $a\Rc b$ so $\varphi$ is not
injective. 

However, we can prove that $\varphi$ is an isomorphism if $\trianglelefteq_{l}$
is contained in a (principally finite) partial order. 
\begin{lem}
\label{lem:containment_in_poset}Let $S$ be a reduced $E$-Fountain
semigroup which satisfies the congruence condition. Assume $\trianglelefteq_{l}\subseteq\preceq$
where $\preceq$ is some principally finite partial order. Then $\varphi$
is an isomorphism of $\Bbbk$-modules. 
\end{lem}
\begin{proof}
Consider the incidence algebra $\Bbbk[\preceq]$. Define $\zeta_{l}\in\Bbbk[\preceq]$
to be the zeta function of $\trianglelefteq_{l}$:
\[
\zeta_{l}(a,b)=\begin{cases}
1 & a\trianglelefteq_{l}b\\
0 & \text{otherwise}
\end{cases}
\]
As $\zeta_{l}(a,a)=1$ for every $a\in A$, we know that $\zeta_{l}$
has an inverse $\zeta_{l}^{-1}$. This means that 
\[
\zeta_{l}^{-1}\star\zeta_{l}(a,b)=\sum_{a\preceq c\preceq b}\zeta_{l}^{-1}(a,c)\zeta_{l}(c,b)=\delta(a,b)=\begin{cases}
1 & a=b\\
0 & a\neq b
\end{cases}
\]
and likewise 
\[
\zeta_{l}\star\zeta_{l}^{-1}(a,b)=\delta(a,b).
\]
Note that 
\[
\varphi(a)=\sum_{b\trianglelefteq_{l}a}C(b)=\sum_{b\preceq a}\zeta_{l}(b,a)C(b).
\]
The inverse of $\varphi$ is given by 
\[
\psi(C(a))=\sum_{b\preceq a}\zeta_{l}^{-1}(b,a)b.
\]
Indeed 
\begin{align*}
\psi(\varphi(a)) & =\psi(\sum_{b\preceq a}\zeta_{l}(b,a)C(b))\\
 & =\sum_{b\preceq a}\zeta_{l}(b,a)\psi(C(b))\\
 & =\sum_{b\preceq a}\zeta_{l}(b,a)\left(\sum_{c\preceq b}\zeta_{l}^{-1}(c,b)c\right)\\
 & =\sum_{c\preceq a}c\sum_{c\preceq b\preceq a}\zeta_{l}^{-1}(c,b)\zeta_{l}(b,a)\\
 & =\sum_{c\preceq a}c\delta(c,a)=a
\end{align*}
and a similar argument shows that $\varphi(\psi(C(a))=C(a)$.
\end{proof}
For future reference we state clearly the following immediate corollary
of \thmref{main_thm_for_now} and \lemref{containment_in_poset}.
\begin{thm}
\label{thm:isomorphism_theorem}If $\trianglelefteq_{l}$ is contained
in some (principally finite) partial order then $\varphi$ is an isomorphism
of $\Bbbk$-algebras if and only if the generalized right ample condition
holds in $S$.
\end{thm}
We will see later that $\trianglelefteq_{l}$ is indeed contained
in a partial order for a few natural cases.
\begin{rem}
The advantage of \thmref{isomorphism_theorem} is that it is often
much easier to study the category algebra $\Bbbk\mathcal{C}$ and
its representations than the original semigroup algebra $\Bbbk S$.
For many natural semigroups the associated category is EI \cite{Margolis2021,Stein2016},
locally trivial \cite{Stein2020} or a poset \cite{Margolis2018A}.
The representation theory of such categories is understood to a certain
extent (see for instance \cite{Li2011,Webb2007}).
\end{rem}

\section{\label{sec:Ehresmann_semigroups}$E$-Ehresmann semigroups}

Recall that a reduced $E$-Fountain semigroup is called $E$-Ehresmann
if the set of idempotents $E$ is a subsemilattice (or a subband,
by \lemref{SubbandImpliesEhresmann}).

We want to show that in the case of $E$-Ehresmann semigroups the
generalized and standard right ample identities are equivalent.
\begin{prop}
\label{prop:In_Subband_RR_Eqiovalent_to_GRR} Let $S$ be an $E$-Ehresmann
semigroup. Then the right ample identity holds if and only if the
generalized right ample identity holds.
\end{prop}
\begin{proof}
In fact, this follows indirectly from \thmref{isomorphism_theorem}
and \cite[Theorem 4.4]{Wang2017} since both properties are equivalent
to $\varphi$ being an isomorphism, but we give a simple direct proof.
In view of \propref{Standard_right_ample_implies_generalised}, it
is enough to show that the generalized right ample identity implies
the standard one. Let $a\in S$ and $e\in E$. If $E$ is a subband
then for every $f\in E$ we have $e\left(a\left(eaf\right)^{\ast}\right)^{+}\in E$
and therefore 
\[
\left(e\left(a\left(eaf\right)^{\ast}\right)^{+}\right)^{\ast}=e\left(a\left(eaf\right)^{\ast}\right)^{+}.
\]
The assumption that $S$ satisfies the generalized right ample identity
implies
\[
e\left(a\left(eaf\right)^{\ast}\right)^{+}=\left(a\left(eaf\right)^{\ast}\right)^{+}
\]
so $e$ is left identity of $\left(a\left(eaf\right)^{\ast}\right)^{+}$.
Since $a(eaf)^{\ast}\Rt(a(eaf)^{\ast})^{+}$this is equivalent to
\[
ea\left(eaf\right)^{\ast}=a\left(eaf\right)^{\ast}
\]
Now we can substitute $f=a^{\ast}$ and obtain
\[
ea=ea(ea)^{\ast}=ea(eaa^{\ast})^{\ast}=a\left(eaa^{\ast}\right)^{\ast}=a\left(ea\right)^{\ast}
\]
as required.
\end{proof}
We have already seen (\corref{EqualityOfLeqs}) that if $S$ is an
$E$-Ehresmann semigroup then $\leq_{l}=\trianglelefteq_{l}$ so $\trianglelefteq_{l}$
is a partial order. In this case \thmref{isomorphism_theorem} is
already known. The fact that $\varphi$ is an isomorphism of $\Bbbk$-algebras
was proved by the author in \cite{Stein2017,Stein2018erratum}. The
necessity of the right ample property was proved by Wang in \cite[Lemma 4.3]{Wang2017}.
In fact, Wang proved a version of \thmref{isomorphism_theorem} for
the class of right $P$-restriction locally Ehresmann $P$-Ehresmann
semigroups (for definitions of these notions, see \cite{Jones2012,Wang2017}).
We leave open the problem of finding a unified generalization for
\thmref{main_thm_for_now} and Wang's result \cite[Theorem 4.4]{Wang2017}.

\section{\label{sec:Catalan_monoid}Catalan monoid }

A function $f:[n]\to[n]$ (where $[n]=\{1,\ldots,n\}$) is called
\emph{order-preserving }if $i\leq j\implies f(i)\leq f(j)$ for every
$i,j\in[n]$ and \emph{order-increasing }if $i\leq f(i)$ for every
$i\in[n]$. Denote by $\CT_{n}$ the monoid of all order-preserving
and order-increasing functions $f:[n]\to[n]$, called the Catalan
monoid. The name comes from the fact that its size is the $n$-th
Catalan number \cite[Theorem 14.2.8]{Ganyushkin2009b}. The Catalan
monoid is $\mathcal{J}$-trivial \cite[Proposition 17.17]{Steinberg2016}
and it is well known that every $\mathcal{J}$-trivial monoid is $E$-Fountain
for $E=E(S)$ \cite[Corollary 3.2]{Margolis2018B}. It is also known
that in a $\mathcal{J}$-trivial monoid $ef=e\iff fe=e\iff e$$\leq_{\Jc}f$
for $e,f\in E(S)$ (\cite[Lemma 3.6]{Denton2010}) so every $\mathcal{J}$-trivial
monoid is in fact a reduced $E$-Fountain monoid. We define a partial
order $\preceq_{n}$ on subsets of $[n]$ in the following way. For
two subsets $X=\{x_{1}<\ldots<x_{k}\}$ and $Y=\{y_{1}<\ldots<y_{r}\}$
we define $X\preceq_{n}Y$ if $k=r$ (i.e., $|X|=|Y|$) and $x_{i}\leq y_{i}$
for every $1\leq i\leq k$. It is well known (see \cite{Margolis2018A})
that elements of $\CT_{n+1}$ are in one-to-one correspondence with
pairs $(X,Y)$ where $X,Y\subseteq[n]$ and $X\preceq_{n}Y$. In other
words, there is a one-to-one correspondence between elements of $\CT_{n+1}$
and elements of $\preceq_{n}$ viewed as a set of ordered pairs. Explicitly,
given such a pair $(X,Y)$ we define 
\[
f_{X,Y}(i)=\begin{cases}
y_{1} & \text{if }1\leq i\leq x_{1}\\
y_{j} & \text{if }x_{j-1}<i\leq x_{j}\\
n+1 & \text{if }x_{k}<i.
\end{cases}
\]
It is easy to verify that $f_{X,Y}\in\CT_{n+1}$. Note that $\im(f_{X,Y})\backslash\{n+1\}=Y$
and that 
\[
\ker(f_{X_{1},Y_{1}})=\ker(f_{X_{2},Y_{2}})\iff X_{1}=X_{2}
\]
(where $\im(f)$ is the image of $f$ and $\ker(f)$ is the equivalence
relation on $[n]$ defined by $\mbox{\ensuremath{(x_{1},x_{2})\in\ker f}}$
if and only if $f(x_{1})=f(x_{2})$). Note also that $x_{i}$ is the
maximal element whose image is $y_{i}$. Recall that $\mbox{\ensuremath{f:[n]\to[n]}}$
is an idempotent if and only if $f(i)=i$ for every $i\in\im(f)$.
Therefore $f_{X,Y}$ is an idempotent if and only if $X=Y$. In particular,
there is a one-to-one correspondence between idempotents in $\CT_{n+1}$
and subsets of $[n]$. For every $Z\subseteq[n]$ we denote by $e_{Z}$
the idempotent corresponding to $Z$. Explicitly (see \cite[Proposition 17.18]{Steinberg2016}):
\[
e_{Z}(i)=\min\{z\in Z\cup\{n+1\}\mid i\leq z\}
\]

The following proposition is part of \cite[Proposition 17.20]{Steinberg2016}
(with notation $f^{-}$ instead of our $f^{\ast}$).
\begin{prop}
\label{prop:Star_Plus_In_Catalan}Let $f_{X,Y}\in\CT_{n+1}$ then
$(f_{X,Y})^{\ast}=e_{X}$ and $(f_{X,Y})^{+}=e_{Y}$.
\end{prop}
As an immediate corollary we have:
\begin{cor}
Let $f_{X_{1},Y_{1}},f_{X_{2},Y_{2}}\in\CT_{n+1}$ then
\[
f_{X_{1},Y_{1}}\Lt f_{X_{2},Y_{2}}\iff X_{1}=X_{2}\iff\ker(f_{X_{1},Y_{1}})=\ker(f_{X_{2},Y_{2}})
\]
\[
f_{X_{1},Y_{1}}\Rt f_{X_{2},Y_{2}}\iff Y_{1}=Y_{2}\iff\im(f_{X_{1},Y_{1}})=\im(f_{X_{2},Y_{2}}).
\]
\end{cor}
\begin{lem}
The Catalan monoid satisfies the congruence condition.
\end{lem}
\begin{proof}
Take two functions $f,g\in\CT_{n+1}$. For $\Lt$ to be a right congruence
we need to show that $(fg)^{\ast}=(f^{\ast}g)^{\ast}$ or equivalently,
$\ker(fg)=\ker(f^{\ast}g)$. Indeed, 
\begin{align*}
(a_{1},a_{2}) & \in\ker(fg)\iff fg(a_{1})=fg(a_{2})\iff(g(a_{1}),g(a_{2}))\in\ker f\\
 & \iff(g(a_{1}),g(a_{2}))\in\ker f^{\ast}\iff f^{\ast}g(a_{1})=f^{\ast}g(a_{2})\\
 & \iff(a_{1},a_{2})\in\ker(f^{\ast}g).
\end{align*}

To show that $\Rt$ is a left congruence we need to show that $(fg)^{+}=(fg^{+})^{+}$
or equivalently, $\im(fg)=\im(fg^{+})$. Indeed,
\begin{align*}
b\in\im(fg) & \iff\exists a\in\im(g),\quad f(a)=b\iff\exists a\in\im(g^{+}),\quad f(a)=b\\
 & \iff b\in\im(fg^{+}).
\end{align*}
\end{proof}
Note that the objects of the associated category are subsets of $[n]$
and if $\mbox{\ensuremath{X,Y\subseteq[n]}}$ then there is a unique
morphism from $X$ to $Y$ if and only if $X\preceq_{n}Y$. Therefore
the associated category is just the poset $\preceq_{n}$ viewed as
a category. 

Our next goal is to prove that the generalized right ample identity
holds in the Catalan monoid. The proof given here is essentially taken
from \cite{Margolis2018A}, but we reformulate it in the language
of $E$-Fountain monoids. We first need the following definition.
\begin{defn}
Let $f\in C_{n+1}$ and let $Z\subseteq[n]$, we say that $Z$ is
a partial cross section ($\PCS$) of $f$ if $n+1\notin f(Z)$ and
$\left.f\right|_{Z}$ is injective (where $\left.f\right|_{Z}$ is
the restriction of $f$ to the set $Z$).
\end{defn}
\begin{prop}
\label{prop:PCS_proposition}Let $f\in C_{n+1}$ and let $Z\subseteq[n]$
then $(fe_{Z})^{\ast}=e_{Z}$ if and only if $Z$ is a $\PCS$ of
$f$.
\end{prop}
\begin{proof}
First assume $(fe_{Z})^{\ast}=e_{Z}$. Let $z\in Z$. Since $e_{Z}$
and $fe_{Z}$ have the same kernel and $e_{Z}(z)=z\neq n+1=e_{Z}(n+1)$
then $f(z)=fe_{Z}(z)\neq fe_{Z}(n+1)=n+1$. Now if $f(z)=f(z^{\prime})$
for some $z,z^{\prime}\in Z$ then $fe_{Z}(z)=fe_{Z}(z^{\prime})$
hence $(z,z^{\prime})\in\ker(fe_{Z})$. Therefore, $(z,z^{\prime})\in\ker(e_{Z})$
so $z=e_{Z}(z)=e_{Z}(z^{\prime})=z^{\prime}$ and $\left.f\right|_{Z}$
is indeed injective. In the other direction assume $Z$ is a $\PCS$
of $f$. We need to prove that $\ker(fe_{Z})=\ker(e_{Z})$. It is
clear that $\ker(e_{Z})\subseteq\ker(fe_{Z})$. For the other containment
assume $(x,x^{\prime})\in\ker(fe_{Z})$ so $fe_{Z}(x)=fe_{Z}(x^{\prime})$.
First note that 
\[
fe_{Z}(x)=fe_{Z}(x^{\prime})=n+1\iff e_{Z}(x)=e_{Z}(x^{\prime})=n+1
\]
because $\im(e_{Z})=Z\cup\{n+1\}$ and $n+1\notin f(Z)$. The other
option is that $e_{Z}(x),e_{Z}(x^{\prime})\in Z$. In this case the
fact that $\left.f\right|_{Z}$ is injective implies that 
\[
fe_{Z}(x)=fe_{Z}(x^{\prime})\implies e_{Z}(x)=e_{Z}(x^{\prime})
\]
so $\ker(fe_{Z})\subseteq\ker(e_{Z})$ as required.
\end{proof}
\begin{cor}
\label{cor:Leq_L_and_PCS}If $h\trianglelefteq_{l}f$ for $h,f\in\CT_{n+1}$
and $h^{\ast}=e_{Z}$ where $Z\subseteq[n]$ then $Z$ is a $\PCS$
of $f$. 
\end{cor}
\begin{proof}
We have $h=fe_{Z}$ so $e_{Z}=h^{\ast}=(fe_{Z})^{\ast}$ and the result
follows by \propref{PCS_proposition}.
\end{proof}
\begin{thm}
The Catalan monoid $\CT_{n+1}$ satisfies the generalized right ample
condition.
\end{thm}
\begin{proof}
We use the equivalent condition \enuref{WRR_def_3} given in \lemref{WRR_equivalent_definitions}.
Assume $h\trianglelefteq_{l}fg$ for some $\mbox{\ensuremath{f,g,h\in\CT_{n+1}}}$
and denote $h^{\ast}=e_{Z}$. Since $\text{\mbox{\ensuremath{h\trianglelefteq_{l}fg}}}$
we know that $Z$ is a $\PCS$ of $fg$ by \corref{Leq_L_and_PCS}.
Denote $\mbox{\ensuremath{W=\im(ge_{Z})\backslash\{n+1\}}}$ and note
that for every $w\in W$ there exists $z\in Z$ such that $\mbox{\ensuremath{g(z)=ge_{Z}(z)=w}}$.
We need to prove that 
\[
(f(ge_{Z})^{+})^{\ast}=(ge_{Z})^{+}.
\]
According to \propref{PCS_proposition} and \propref{Star_Plus_In_Catalan}
we need to show that $W$ is a $\PCS$ of $f$. Take $w\in W$ and
let $z\in Z$ such that $g(z)=w$, then $f(w)=fg(z)\neq n+1$ since
$Z$ is a $\PCS$ of $fg$. Now take $w,w^{\prime}\in W$ such that
$f(w)=f(w^{\prime})$. Take $z,z^{\prime}\in Z$ such that $g(z)=w$
and $g(z^{\prime})=w^{\prime}$. Then $fg(z)=fg(z^{\prime})$ so $z=z^{\prime}$
since $\left.fg\right|_{Z}$ is injective.
\end{proof}
Since its idempotents do not form a subband (for $n\geq3$), the Catalan
monoid $\CT_{n}$ is another example of a semigroup which satisfies
the generalized right ample condition but not the standard one. A
concrete simple counter-example is given in the following example:
\begin{example}
\label{exa:Right_restriction_Counter_example}Consider the idempotents
$e_{1},e_{2}\in\CT_{3}$ defined by
\[
e_{1}(i)=\begin{cases}
2 & i=1,2\\
3 & i=3
\end{cases},\quad e_{2}(i)=\begin{cases}
1 & i=1\\
3 & i=2,3
\end{cases}
\]
so
\[
e_{1}e_{2}(i)=\begin{cases}
2 & i=1\\
3 & i=2,3.
\end{cases}
\]
It is easy to see that $(e_{1}e_{2})^{\ast}=e_{2}$ as $\ker e_{2}=\ker e_{1}e_{2}$.
Therefore,
\[
e_{2}(e_{1}e_{2})^{\ast}=e_{2}e_{2}=e_{2}\neq e_{1}e_{2}
\]
so the right ample condition does not hold.
\end{example}
\begin{rem}
\label{rem:Leq_l_not_transitive}\exaref{Right_restriction_Counter_example}
also shows that the relation $\trianglelefteq_{l}$ in $\CT_{n}$
is not transitive and hence not even a preorder. It is clear that
$e_{1}\trianglelefteq_{l}\id$ (where $\id$ is the identity function)
and $e_{1}e_{2}\trianglelefteq_{l}e_{1}$. However, $e_{1}e_{2}$
is not an idempotent hence $e_{1}e_{2}\ntrianglelefteq_{l}\id$. 
\end{rem}
Fortunately, it is easy to see that $\trianglelefteq_{l}$ is contained
in a partial order.
\begin{lem}
Let $S$ be a reduced $E$-Fountain semigroup. If $S$ is $\mathcal{R}$-trivial
then $\trianglelefteq_{l}$ is contained in a partial order. 
\end{lem}
\begin{proof}
If $a\trianglelefteq_{l}b$ then $a=ba^{\ast}$ so $a\leq_{\Rc}b$
hence $\trianglelefteq_{l}\subseteq\leq_{\Rc}$. The relation $\leq_{\Rc}$
is a partial order if and only if $S$ is $\mathcal{R}$-trivial.
\end{proof}
In particular, $\leq_{\Rc}$ is a partial order in the Catalan monoid
which is a $\Jc$-trivial monoid. In conclusion, the results in this
section and \thmref{isomorphism_theorem} imply the following corollary:
\begin{thm}
\label{thm:Catalan_isomorphism}Let $\Bbbk$ be a unital commutative
ring. There is an isomorphism of algebras $\mbox{\ensuremath{\Bbbk\CT_{n+1}\simeq\Bbbk[\preceq_{n}]}}$.
In particular, $\Bbbk\CT_{n+1}$ is an incidence algebra.
\end{thm}
\begin{rem}
This isomorphism is precisely the one given in \cite[Theorem 3.1]{Margolis2018A}.

As a final observation, we consider also the generalized left ample
condition. 
\end{rem}
\begin{prop}
The Catalan monoid $\CT_{n+1}$ satisfies the generalized left ample
condition.
\end{prop}
\begin{proof}
Note that $(e_{Z}g)^{+}=e_{Z}$ means that $e_{Z}g$ and $e_{Z}$
have the same image. In other words, for every $z\in Z$ there exists
$x$ such that $e_{Z}g(x)=z$. In this case we say that $g$ is a
multi cross section ($\MCS$) of $Z$. We use the dual of condition
\enuref{WRR_def_3} given in \lemref{WRR_equivalent_definitions}:
\[
h\trianglelefteq_{r}fg\implies((h^{+}f)^{\ast}g)^{+}=(h^{+}f)^{\ast}
\]
If $h\trianglelefteq_{r}fg$ and $h^{+}=e_{Z}$ then $e_{Z}=h^{+}=(h^{+}fg)^{+}=(e_{Z}fg)^{+}$
so $fg$ is an $\MCS$ of $Z$. Denote $e_{W}=(e_{Z}f)^{\ast}$. We
need to prove that $g$ is an $\MCS$ of $W$. Let $w\in W$, we need
to find $x$ such that $e_{W}g(x)=w=e_{W}(w)$. Note that $e_{W}$
and $e_{Z}f$ have the same kernel so this is equivalent to finding
$x$ such that $e_{Z}fg(x)=e_{Z}f(w)$. Set $z=e_{Z}f(w)$ so $z\in\im(e_{Z})$.
If $z=n+1$ then $e_{Z}f(w)=e_{Z}f(n+1)$ so $w=e_{W}(w)=e_{W}(n+1)=n+1$,
a contradiction. Therefore, $z\in\im(e_{Z})\backslash\{n+1\}=Z$.
Since $fg$ is an $\MCS$ of $Z$ there exists an $x$ such that $e_{Z}fg(x)=z$
as required.
\end{proof}
\bibliographystyle{plain}
\bibliography{library}

\end{document}